\documentclass{emsprocart}
\usepackage{mathrsfs}
 \usepackage[all]{xy}

\contact[joseph.chuang@city.ac.uk]{Joseph Chuang, Department of Mathematics,  City,  University of  London, EC1V0HB, London, U.K.}

\contact[radha.kessar.1@city.ac.uk]{ Radha Kessar, Department of Mathematics,  City, University  of London, EC1V0HB, London, U.K.  }

\newtheorem{thm}{Theorem}
\newtheorem{lem}[thm]{Lemma}
\newtheorem{pro}[thm]{Proposition}
\newtheorem{cor}[thm]{Corollary}

\theoremstyle{definition}

\newtheorem{defi}[thm]{Definition}

\title
{On Perverse Equivalences and Rationality}
\begin{document}
\author{Joseph Chuang, Radha Kessar}

\begin{abstract}
We show that  perverse equivalences between  module categories of  finite-dimensional algebras   preserve   rationality.    As an application, we     give   a  connection between   some famous   conjectures  from the modular representation theory  of finite   groups,  namely   Brou\'e's Abelian Defect Group conjecture  and Donovan's Finiteness  conjectures.
\end{abstract}


\maketitle

\bigskip

Let $k$ be an  algebraically  closed field  and let   $A$ be a finite dimensional   $k$-algebra.   
Denote by   $k$-$\mathrm{vect}$ the category  whose objects are finite dimensional   $k$-vector spaces    and   whose   morphisms  are    $k$-linear  transformations between vector spaces,   by   $A$-$\mathrm{mod}$ the $k$-linear category of  finitely  generated (left)  $A$-modules, and by  $D^b(A) $ the bounded derived category of finitely generated     $A$-modules. 
Recall that  $D^b(A) $ is  a $k$-linear,  triangulated  category.   
 
In this article we will   be dealing   with  additive, but possibly non $k$-linear    functors between $k$-linear   categories.   Thus,  if $\mathscr{ F} \colon {\mathcal A} \to  {\mathcal B}  $ is a  functor   between  $k$-linear categories,  we will   specify  whether $ \mathscr{ F}  $ is $k$-linear or merely additive.

Let $\sigma \colon  k \to    k $ be  a field automorphism.   
If $V$ is a   $k$-vector space,  the   $\sigma $-twist $V^{\sigma}$  of $V$   is   the $k$-vector space which is equal to  $V$ as a group,  but where  scalar   multiplication   is given by  $\lambda\cdot x=\sigma^{-1} (\lambda)  x$.  Any $k$-linear map $f\colon  V \to W $is also a $k$-linear map $f\colon  V^{\sigma} \to W^{\sigma}$.   Denote also  by   $\sigma\colon  k$-$\mathrm{vect} \to k$-$\mathrm{vect} $ the   functor which   sends an object $V$ to $V^{\sigma} $  and which is the identity on morphisms.  Then $\sigma $ is an  additive   equivalence.
We denote   by $A^{\sigma}$  the $k$-algebra     which equals  
$ A^{\sigma}$  as a $k$-vector space  and $A$ as a  ring.       The  functor 
$\sigma $ on $k$-$\mathrm{vect}$ extends to an   additive  equivalence   $\sigma\colon  A$-$\mathrm{mod}   \to  A^{\sigma}$-$\mathrm{mod}$.  However, $A$ and $A^{\sigma} $  are not necessarily  Morita equivalent   as $k$-algebras.

\begin{defi}\label{maindefinition} A $k$-linear equivalence $\mathscr{E}\colon A$-$\mathrm{mod} \to  A^{\sigma}$-$\mathrm{mod}  $   is said to be a  \emph{ $\sigma$-Morita equivalence}    if $ \mathscr{E}(V)  \cong    V^{\sigma}    $ for  all simple $ A$-modules $V$.  If there is a   $\sigma$-Morita  equivalence between $A$ and $A^{\sigma}$, then we say that   $A$ and $A^{\sigma} $ are   \emph{$\sigma$-Morita equivalent}.
\end{defi}
Note that  if $ \mathscr{E}\colon A$-$\mathrm{mod} \to  A^{\sigma}$-$\mathrm{mod}  $   is a $\sigma $-Morita equivalence, then  $\mathscr{E}$  induces a dimension preserving bijection between the set of  isomorphism classes of   simple $A$-modules and the set of isomorphism classes   of  simple 
$A^{\sigma}$-modules,  whence $ A$ and $ A^{\sigma}$ are   isomorphic as $k$-algebras.

 Perverse equivalences  were introduced  by R. Rouquier and the first author in  \cite{CR}.   For a finite dimensional  $k$-algebra  $A$,  denote by    ${\mathcal   S}_A $    a set of representatives for  the  isomorphism classes   of  simple $A$-modules.     A $k$-linear equivalence  of triangulated categories  $\mathscr {F} \colon  D^b(A) \to D^b(B) $, for $B$ a finite dimensional algebra,  is  said to be a \emph{perverse 
equivalence}  if the following holds : 
  
There exists a  non-negative integer   $r$, a function $ q \colon [0,r] \to {\mathbb Z} $,  a    filtration $\emptyset = {\mathcal  S}_{-1} \subset {\mathcal S}_0 \subset \cdots \subset {\mathcal S}_r = \mathcal{S}_A $   and a filtration  $\emptyset = {\mathcal  S}'_{-1} \subset {\mathcal S}'_0 \subset \cdots \subset {\mathcal S}'_r = \mathcal{S}_B$  such that  whenever  $ T  $ is in ${\mathcal S}_i $, the composition factors of $H^{-j}   (\mathscr{F}(T))    $ are in   ${\mathcal S}'_{i-1} $  for $ j\ne q(i) $  and  in ${\mathcal S}'
_i $  for $j=q(i) $ (see   \cite[Definition~3.1.1]{CrR}).    The tuple $ (r, q, {\mathcal S}, {\mathcal S'}) $ is       then  called a  perversity datum for $\mathscr{F}$.    The algebras $ A$ and $B$ are {\it perversely equivalent}   if   there  exists a perverse equivalence   $ \mathscr{F}\colon  D^b(A) \to D^b(B) $.

Trivial examples of perverse equivalences arise from Morita equivalences. Indeed, if
$\mathscr{F}\colon  A$-$\mathrm{mod}   \to  B$-$\mathrm{mod}$ is a $k$-linear equivalence, then the induced equivalence  $\mathscr{F}\colon D^b(A) \to   D^b(B)$ is a perverse equivalence with respect to
the perversity datum $(r,q_0,{\mathcal S}, {\mathcal S}')$, where $r$ is any non-negative integer, ${\mathcal S}$ is any filtration on ${\mathcal S}_A$, as above,  ${\mathcal S}'$ is the filtration on ${\mathcal S}_B$ induced by $\mathscr{F}$ and $q_0\colon [0,r]\to\mathbb{Z}$ is defined by $q(i)=0$ for all $i$.
The converse statement is the first part of the following lemma, 
which records basic properties of perverse equivalences.
\begin{lem}\cite[Lemma~3.69]{CR}\label{basiclemma}
	Let $A$ and $B$ be finite-dimensional $k$-algebras, and let
	 $\mathscr{F}\colon D^b(A) \to   D^b(B)$ be a perverse equivalence with respect to the 
	 perversity datum  $ (r, q, {\mathcal S}, {\mathcal S}') $.
	\begin{enumerate}
		\item If $q=q_0$, then
		$\mathscr{F}$ restricts to a $k$-linear equivalence $A$-$\mathrm{mod}\to B$-$\mathrm{mod}$.
			\item  Any inverse equivalence $\mathscr{F}^{-1}\colon D^b(B) \to   D^b(A)$ 
			 is a perverse equivalence with respect to the
			perversity datum  $ (r, -q, {\mathcal S}', {\mathcal S}) $.
			\item If $C$ is a finite-dimensional $k$-algebra, and 
			$\mathscr{G}\colon D^b(B) \to   D^b(C)$ is a perverse equivalence with respect to  the 
			perversity datum  $ (r, q', {\mathcal S'}, {\mathcal S}'') $, then the composition $\mathscr{G}\circ \mathscr{F}\colon D^b(A) \to   D^b(C)$ is a perverse equivalence
			with respect to the perversity datum $(r,q'+q', {\mathcal S}, {\mathcal S}'')$.
		\end{enumerate}
	\end{lem}
Note that   the functor $\sigma\colon   A$-$\mathrm{mod}   \to  A^{\sigma}$-$\mathrm{mod}$  extends to an  equivalence  of triangulated categories  $ \sigma \colon D^b(A) \to D^b(A^{\sigma} ) $  (again, not necessarily $k$-linear).   Further, if $B$ is a finite dimensional $k$-algebra  and  $ \mathscr{F}\colon  D^b(A) \to   D^b(B)  $   is a   $k$-linear  exact functor,  then $  \sigma \circ   \mathscr{F} \circ  \sigma^{-1}  \colon    D^b(A^{\sigma} )   \to     D^b(B^{\sigma} )   $   is $k$-linear and exact.  If    $ \mathscr{F} \colon    D^b(A) \to D^b(B) $     is  a perverse equivalence with respect to perversity datum $ (r, q, {\mathcal S}, {\mathcal S}') $, then  we   let  $ {\mathcal S^{\sigma}  }  $  and    
${\mathcal S'}^{\sigma}$   be  the   filtrations defined by    $  T^{\sigma}\in  {\mathcal S_i^{\sigma} }  $ if and only if $T \in{\mathcal S}_i $ and  $  T^{'\sigma}\in  {\mathcal S_i^{'\sigma} } $ if and only if 
$T' \in{\mathcal S}_i  $, $ -1  \leq i \leq r $.  The following lemma is an  immediate consequence  of  the definitions.

\begin{lem}   Let   $ A$ and  $B$ be finite dimensional $k$-algebras  and  $ \sigma \colon  k \to k $ a  field automorphism. If   $ \mathscr{F} \colon    D^b(A) \to D^b(B) $     is  a perverse equivalence with respect to the  perversity datum $ (r, q, {\mathcal S}, {\mathcal S}') $, then 
$$ \sigma \circ   \mathscr{F} \circ  \sigma^{-1}  \colon    D^b(A^{\sigma} )   \to     D^b(B^{\sigma} )        $$  is a perverse   equivalence with   respect to the perversity datum     
$ (r, q, {\mathcal S^{\sigma} }, {\mathcal S'}^{\sigma}) $. 
\end{lem} 

For a finite dimensional $k$-algebra    $A$,    denote by $K_0(A) $  the   Grothendieck group of  
finitely generated  $A$-modules with respect to short exact sequences and for a finitely generated $A$-module $X$ denote by  $ [X] $ the  equivalence class  of $X$ in $K_0(A)$;  $K_0(A) $    is an abelian group, freely generated by $[V]$, $V\in{ \mathcal S}_A$.
If $B$ is a finite-dimensional $k$-algebra and $\mathscr{F}\colon  D^b(A) \to D^b(B) $  is   an  exact  functor  (not necessarily $k$-linear), then   we denote by $[\mathscr{F}]\colon  K_0(A) \to K_0(B) $    the  induced   homomorphism, defined by  
$[\mathscr{F}] ([X]) =  \sum_i (-1)^i [H^i( \mathscr{F}(X) )] $.  We may interpret $[\mathscr{F}]$ as the homomorphism induced by $\mathscr{F}$ between the Grothendieck groups of the triangulated categories $D^b(A)$ and $D^b(B)$, once they are identified with $K_0(A)$ and $K_0(B)$ in the standard way. Hence if $\mathscr{G}\colon  D^b(B) \to D^b(C) $  is  also exact, for a finite dimensional $k$-algebra $C$,   then $[\mathscr{G}\circ \mathscr{F}] =[\mathscr{G}]\circ[\mathscr{F}] $.

The following observation is the  main  result of this note.  
\begin{pro} \label{pro:main} Let $A$ and $B $  be finite dimensional  $k$-algebras and $\sigma\colon  k \to k $  a field automorphism.   If   $A$ and $B$ are perversely equivalent and  $  B$ and $ B^{\sigma} $ are   $\sigma$-Morita equivalent, then  $ A$ and 
$ A^{\sigma}$ are   $\sigma$-Morita equivalent.     
\end{pro}

\begin{proof}     Let     $ \mathscr{F} \colon  D^b(A) \to D^b(B) $    be    a perverse equivalence with respect to perversity datum $ (r, q, {\mathcal S}, {\mathcal S'}) $  and  let $ \mathscr{E} \colon  B$-$\mathrm {mod} \to  B^{\sigma}$-$\mathrm {mod}  $  be a   $\sigma$-Morita equivalence.  Denote also by 
$\mathscr{E} \colon  D^b(B) \to  D^b(B^{\sigma}) $  the  equivalence  induced by $\mathscr{E}$.    Then the hypothesis  on $\mathscr{E}$ implies that     
$\mathscr{E}$   is perverse with respect to the datum  $(r,  q_0,  {\mathcal S}',  {\mathcal S'}^{\sigma})$. 
So by the  above  two lemmas,  the composition 
$$   (\sigma \circ  \mathscr{F} \circ  \sigma^{-1} )^{-1}  \circ  \mathscr{E}\circ  \mathscr{F}\colon    D^b(A )   \to     D^b( A^{\sigma} )  $$  is a perverse equivalence with respect to the datum   $(r, q_0,  {\mathcal S},   {\mathcal S} ^{\sigma}  ) $.   By Lemma~\ref{basiclemma},  it follows that 
 $  (\sigma \circ  \mathscr{F}\circ  \sigma^{-1}) ^{-1}\circ  \mathscr{E}\circ  \mathscr{F}  $  restricts to   a   $k$-linear  
equivalence   from 
 $ A$-$\mathrm{mod} $    to $   A^{\sigma}$-$\mathrm{mod}   $.    

 It remains to show   that  $(\sigma \circ   \mathscr{F} \circ  \sigma^{-1}) ^{-1}\circ  \mathscr{E}\circ \mathscr{F}$  is a  $\sigma$-Morita equivalence.   
    Let $V$ be a  simple $A$-module. Then    $(\sigma \circ   \mathscr{F} \circ  \sigma^{-1}) ^{-1}\circ  \mathscr{E} \circ  \mathscr {F} (V)  $  is  a  simple $ A^{\sigma} $-module, hence is completely determined   by  
$$  \sum_{i} (-1)^i [H^i((\sigma \circ   \mathscr{F} \circ  \sigma^{-1}) ^{-1}\circ  \mathscr{E}\circ  \mathscr{F} (V))] =
 [( \sigma \circ   \mathscr{F} \circ  \sigma^{-1}) ^{-1}\circ  \mathscr{E}\circ \mathscr{F}  ]([V] ).      $$
Each of the functors   $ \mathscr{E} $, $\mathscr{F}$  and  $\sigma$ is     exact,  and by hypothesis $[\mathscr{E}]= [\sigma]  $. Hence   
$$  [ (\sigma \circ   \mathscr{F} \circ  \sigma^{-1}) ^{-1}\circ  \mathscr{E}\circ  \mathscr{F}  ]=   
[\sigma]\circ [\mathscr{F}]^{-1} \circ [\sigma]^{-1} \circ[\mathscr{E}]\circ [\mathscr{F}]= [\sigma].$$    
Thus,   $ (\sigma \circ   \mathscr{F} \circ  \sigma^{-1}) ^{-1}\circ  \mathscr{E}\circ  \mathscr{F} (V)  \cong   V^{\sigma} $.  The following commutative diagram may help the reader. 
\begin{displaymath} \xymatrix{ D^b(A)  \ar[r]^{\mathscr{F}} \ar[d] & D^b(B) \ar[d]^{\mathscr{E}} \\ 
D^b(A^{\sigma}) \ar[r]^{\sigma \circ \mathscr{F} \circ \sigma^{-1}} &D^b(B^{\sigma}) }
\end{displaymath}
\end{proof}  

From now on let $p$ be a prime number. Suppose  that $ k  =\bar {\mathbb{ F}}_p $   is an algebraic closure of   the field of 
$p$  elements  and   that $\sigma\colon  k \to k $ is   the Frobenius automorphism $ \lambda \to \lambda^p$. Recall from \cite{BK} that   the {\it  Morita Frobenius number}   of a finite dimensional  $k$-algebra    $A $ is the least positive integer $m$ such that   $ A$  and   $A^{\sigma^m} $ are Morita equivalent (as $k$-algebras).    By \cite{Ke}, the Morita Frobenius number  of $A$  is also the least  positive integer $m$ such 
that   a basic algebra of $A$  has an ${\mathbb F}_{p^m} $-form. Thus the  Morita Frobenius number   of $A$ is a measure of the rationality of $A$.

  As   a variation on  the  theme  of Morita Frobenius numbers,    for a  finite dimensional $k$-algebra  $A$, we    define the    {\it $\sigma$-Morita  Frobenius  number of  $A$ }  to be the least positive number  $m$ such that $ A$ and $A^{\sigma ^m} $ are $\sigma^m $-Morita equivalent.  Note that   some finite  subfield  of $k$  is a splitting   field  for  $A$,   whence  the  $\sigma$-Morita  Frobenius number     is defined. Also, clearly, 
the   Morita Frobenius number of $A$ is always less than or equal to  the $\sigma $-Morita Frobenius   number  of $A$.  As an immediate  consequence of   Proposition~\ref{pro:main}   we obtain the following.

\begin{cor}\label{cor:main}    Let $ k$ and $\sigma $ be as above and let  $ A$ and $B$ be   
finite  dimensional  $k$-algebras. If $A$ and $B$ are perversely equivalent, then   the $\sigma $-Morita  Frobenius number of $A$  is equal to  the $\sigma $-Morita  Frobenius  number of $B$.  
\end{cor}

We remark that there  are no known examples  of perversely equivalent (or even derived equivalent) algebras   which do   not have the  same Morita Frobenius  number.   

The interest in Morita Frobenius numbers comes in part from the   finiteness conjectures  of Donovan 
 in  the local representation theory of finite groups (see \cite[Conjecture M]{Alp}).      Let   $P$  be  a  finite $p$-group.  By a $P$-block   we mean     a block  $A$   of the group algebra $kG$,     $G$ a finite group,   such that the defect groups of $A$ are isomorphic to $P$. Donovan's conjecture  states  that  the number  of Morita equivalence classes  of  $P$-blocks is  bounded   by a function   that depends  only on the order of  $P$.     A weak version of  Donovan's conjecture 
states   that  the  entries of the  Cartan matrix   of      $P$-blocks  are  bounded   by a function   which depends  only on the order  of $P$.  Both conjectures   were  inspired    by    Brauer's problem 22 (\cite{Br},  see also \cite{La}).  
In  \cite{Ke}   the second author showed  that the gap between  Donovan's two  conjectures   is equivalent  to  the   statement that   the  Morita Frobenius numbers of  $P$-blocks are bounded by a function of $|P|$. This is now known as the rationality conjecture. 

Let $G$ be a finite group, $ P\leq G$,  $A$ be a  $P$-block of $kG$   and $B$ be  the block of $kN_G(P)$       in  Brauer correspondence with $B$.   Recall  that  Brou\'e's abelian defect group conjecture   states that  if $P$ is abelian, then $ D^b(A) $ and  $ D^b(B)$ are  equivalent (as $k$-linear triangulated categories).     In \cite{CR}   it is shown     that  many  known  instances    of derived equivalences between  $A$ and $B$ are compositions of   perverse equivalences, for instance if $P$ is cyclic.   It is further  conjectured that    if $G$ is a finite group of Lie type   in characteristic   $r \ne p$,  and $P$  is abelian, then the derived equivalence between
$A$ and $B$ predicted by Brou\'e \cite[\S6]{B} to arise from the  complex of  cohomology  of Deligne-Lusztig varieties should be perverse (see \cite{CR}, \cite[\S3.4.1]{CrR}, \cite[Conjecture 1.3]{C}).

The existence of (chains of) perverse equivalences  between  blocks and their Brauer correspondents     has  the following  consequence  for     Donovan's conjecture.

\begin{thm}  Let   $k  =\bar {\mathbb F}_p$,  $G$ be a finite  group, $A$ be a block of $kG$, $P$ be a defect group of  $A$ and $B$ be   the block of $kN_G(P) $ in Brauer correspondence with $A$. Suppose that    there exist    finite dimensional    $k$-algebras  
$A_0:=    A, A_1,  \cdots,  A_n  =: B $    such       that   $A_{i-1} $  is perversely equivalent to   
$A_i $,  $ 1\leq i \leq n $. Then the Morita Frobenius number of $A$      is    at most $(| \mathrm {Out}(P)|_{p'})^2 $, where $ \mathrm {Out }(P) $ denotes the outer automorphism  group of $P$ and 
$ | \mathrm {Out}(P)|_{p'} $ denotes  the $p'$-part of the order of  $\mathrm {Out }(P) $.
\end{thm}

\begin{proof}   Let $ a=  | \mathrm {Out}(P)|_{p'} $.   Since the Morita Frobenius number  of $A$ is at most the $\sigma$-Morita number of $A$,     by  (repeated)  applications of  
Corollary  \ref{cor:main} it suffices to show that the  $\sigma$-Morita Frobenius number    of $ B$ is at most $a^2$.  By  K\"ulshammer's  structure theorem   for blocks with normal defect group \cite{Ku}, there exists  a  subgroup $ E \leq {\rm Aut }(P)$   of $p'$-order    and an element $ \alpha $ of $H^2( P\rtimes E, k^{\times} )$ such that  $ B$ is Morita equivalent    to the twisted group algebra $ k_{\alpha} (P \rtimes E) $ (as $k$-algebras). 
So, again by Proposition~\ref{cor:main}   it suffices to prove that  the  $\sigma$-Morita Frobenius number  of  $k_{\alpha} (P \rtimes E) $ is   at most $a^2$.    Note that since $E$ is a $p'$-group,  $E$ is isomorphic to a subgroup of 
$\mathrm{Out}(P)$    and  consequently $|E| \leq  a $.

Now, $k_{\alpha} (P \rtimes E) $ is isomorphic as $k$-algebra to    an algebra of the form $ k  F  \tilde c $, where $ F$  is a    central extension of $P\rtimes E$ by a cyclic group, say $Z$  of order $|E|$,   and $ \tilde c  \in    k Z   $ is a central idempotent of  $kF $ (see for example \cite[\S10.5]{Th}).       

Let $ k_0  \subset k $ be the splitting field of $x^{|E|^2 -1} $  and   let $a'= |k:{\mathbb F}_p|$.   Then,  $\tilde c  \in k_0 Z$  and 
$ k F \tilde c =  k  \otimes_{k_0}   k_0 F \tilde c $. Since  $ P$ is normal in $F $, every simple  
$k F $-module  is the inflation of a simple $ k (F/P)$-module. On the other hand,  since $|F/P|=  |E|^2 $,  $k_0$  is a splitting 
field of   $F/P$.    Thus,  $k_0   F \tilde c $
is  a  split   $k_0$-algebra.   It follows that  $kF\tilde c  $ and $ (kF \tilde c) ^{\sigma^{a'}}  $ are  $\sigma^{a'}$- Morita equivalent.   The result follows as  $a' \leq a^2 $.
\end{proof}

{\bf Remark.}   The Morita Frobenius numbers  of almost  all   unipotent blocks of quasi-simple finite groups  of Lie type in 
non-describing characteristic   have been determined  in \cite{F}.    We hope that  the approach to Morita Frobenius numbers via   perverse equivalences   outlined   above   will help  in settling the remaining     cases.

\

{\bf Acknowldgement.} We would like to  thank the referees   for their valuable suggestions  for improving the exposition  of the paper.

\end{document}